\documentclass[12pt]{article}
\usepackage{amsmath, amsfonts, amsthm,amssymb, color, enumerate, extarrows, xfrac, indentfirst, hyperref, caption,  nicefrac}

\usepackage{cite}

\hypersetup{
	colorlinks   = true,
	citecolor    = magenta}

\usepackage{enumitem}
\setlist[description]{leftmargin=2.62cm, labelindent=1cm}
\newtheorem{theorem}{Theorem}[section]
\newtheorem{definition}[theorem]{Definition}

\newtheorem{lem}[theorem]{Lemma}
\newtheorem{pro}[theorem]{Proposition}

\DeclareMathOperator{\Real}{Re}

\DeclareMathOperator{\Tr}{Tr}
\DeclareMathOperator*{\esssup}{ess\,sup}

\numberwithin{equation}{section}

\newtheorem{example}{Example}

\usepackage[titletoc]{appendix}

\def \d{\partial}
\def\zbar{\overline{z}}
\def\dbar{\overline{\partial}}
\def\wbar{\overline{w}}

\def\wbar{\overline{w}}

\def\Xbar{\overline{X}}

\def\Xbar{\overline{X}}

\def\kbar{\overline{k}}

\def\CC{{\rm\kern.24em\vrule
width.02em height1.4ex
depth-.05ex\kern-.26em C}}

\def\1bar{\overline{1}}
\def\2bar{\overline{2}}

\def\Omegabar{\overline{\Omega}}

\begin{document}
\title{\bf Sharp pointwise and uniform estimates for $\bar\partial$ \rm}
\author{Robert Xin Dong$^*$, \   Song-Ying Li \  and \ John N. Treuer}
\date{
}
\maketitle

\renewcommand{\thefootnote}{\fnsymbol{footnote}}
\footnotetext{\hspace*{-7mm} 
\begin{tabular}{@{}r@{}p{16.5cm}@{}}
& $^*$Corresponding author\\
& Keywords. $L^2$ minimal solution, canonical solution, $\bar{\partial}$ equation, Cartan classical domain, Bergman kernel, Bergman metric\\
& Mathematics Subject Classification. Primary 32A25; Secondary 32M15, 32W05, 32A36
\end{tabular}}

\begin{abstract}
We use weighted $L^2$-methods to obtain sharp pointwise estimates for the canonical solution to the equation $\bar\partial u=f$ on smoothly bounded strictly convex domains and the Cartan classical domain domains when $f$ is bounded in the Bergman metric $g$. We provide examples to show our pointwise estimates are sharp. In particular, we show that on the Cartan classical domains $\Omega$ of rank $2$ the maximum blow up order is greater than $-\log \delta_\Omega(z)$, which was obtained for the unit ball case by Berndtsson. For example, for IV$(n)$ with $n \geq 3$, the maximum blow up order is $\delta(z)^{1 -{n \over 2}}$ because of the contribution of the Bergman kernel.
Additionally, we obtain uniform estimates for the canonical solutions on the polydiscs, strictly pseudoconvex domains and the Cartan classical domains under stronger conditions on $f$.
\end{abstract}

\section{Introduction}
\indent
The existence and regularity of solutions to the Cauchy-Riemann equation $\dbar u = f$ on a bounded pseudoconvex domain $\Omega$ in $\mathbb{C}^n$ 
is  a fundamental topic  in Several Complex Variables. Since the kernel of $\dbar$ is the set of holomorphic functions, a solution to the Cauchy-Riemann equation 
is not unique if it exists. However, let $A^2(\Omega) := L^2(\Omega) \cap \hbox{ker}(\dbar)$ denote the Bergman space over $\Omega$; then, the solution to $\dbar u=f$ with $u\perp A^2(\Omega)$ is unique, and it is 
called the canonical solution or $L^2$-minimal solution because it has minimal $L^2$-norm among all solutions.  H\"ormander  \cite{Hor65} showed that if $\Omega$ is bounded and pseudoconvex and $f \in L^2_{(0, 1)}(\Omega)$ is $\dbar$-closed, then there exists a solution $u$ that satisfies the estimate
$
\|u\|_{L^2} \leq C\|f\|_{L^2}
$
for some constant $C$ depending only on the diameter of $\Omega$. In view of H\"ormander's result, a natural question is: does there exist a constant $C$ depending only on $\Omega$ such that for any $\dbar$-closed $f \in L^{\infty}_{(0, 1)}(\Omega)$, there exists a solution to $\dbar u = f$ with $\|u\|_{{\infty}} \leq C\|f\|_{\infty}$? If the answer is yes, we say the $\bar\partial$-equation can be solved with  uniform estimates on $\Omega$. An important method for solving the $\bar\partial$-equation is the integral representation for solutions.  In this method, one constructs a differential form $B(z, w)$ on $\Omega \times \Omega$ which is an $(n, n-1)$  form in $w$  such that solutions to 
$\dbar u = f$  can be written as
\begin{equation} \label{kernel}
u(z)=\int_{\Omega}  B(z, w) \wedge f(w).
\end{equation}
The method of integral representation of solutions was initiated by Cauchy, Leray, Fantappi\'{e}, etc. On a smoothly bounded strictly pseudoconvex domain $\Omega$ in $\mathbb{C}^n$, Henkin \cite{Hen70}, Grauert and Lieb \cite{G-L70} constructed integral kernels $B(z, w)$ such that $u$ given by \eqref{kernel} satisfies $\|u\|_{{\infty}} \leq C\|f\|_{\infty}$. Kerzman \cite{Ker71} improved the estimate by showing that $\|u\|_{C^{\alpha}(\Omegabar)} \leq C_{\alpha}\|f\|_{\infty}$ for any $0 < \alpha < \sfrac{1}{2}$. Henkin and Romanov \cite{HR} obtained the sharp estimate $\|u\|_{C^{\sfrac{1}{2}}(\Omegabar)} \leq C\|f\|_{\infty}.$ For more results on strictly pseudoconvex domains, the reader may consult the papers by Krantz \cite{K76}, Range and Siu \cite{R-S72, R-S73} and the books by Chen and Shaw \cite{CS}, Forn\ae ss and Stens\o nes \cite{FS87}, and Range \cite{R86}.
\bigskip

When the class of domains under consideration is changed from strictly pseudoconvex to weakly pseudoconvex, it is no longer possible to conclude in generality the  existence of uniform estimates for $\dbar$.  Berndtsson \cite{Be93}, Forn\ae ss \cite{For86} and Sibony \cite{Sib80} constructed examples of weakly pseudoconvex domains in $\mathbb{C}^2$ and $\mathbb{C}^3$ where uniform estimates for $\bar\partial$ fail.  More strikingly, Forn\ae ss and Sibony \cite{FS93} constructed a smoothly bounded pseudoconvex domain $\Omega \subset \mathbb{C}^2$ such that $\partial \Omega$ is strictly pseudoconvex except at one point, but any solution to $\dbar u = f$ for some given  $\dbar$-closed $f \in L^{\infty}_{(0, 1)}(\Omega)$ does not belong to $L^p(\Omega)$ for any $2 < p \leq \infty$. Range in \cite{R78} gave uniform estimates on bounded convex domains in $\mathbb{C}^2$ with real-analytic boundaries, and in \cite{R90} gave H\"{o}lder estimates on pseudoconvex domains of finite type in $\mathbb{C}^2$.  See  \cite{LTL93, MS99} for related results. Of particular interest is the unit polydisc $\mathbb D^n:=\mathbb D(0,1)^n \subset \mathbb{C}^n$, which is pseudoconvex with non-smooth boundary. When $n=2$, Henkin in \cite{Hen71} showed that there exists a constant $C$ such that
$
\|u\|_\infty \le C \|f\|_\infty
$
for any $f\in C^1_{(0,1)}(\overline{\mathbb D^2})$. Landucci in \cite{L75} obtained the same uniform estimate for the canonical solution on $\mathbb D^2$. Chen and McNeal \cite{CM19}, Fassina and Pan \cite{FP19} generalized Henkin's result to higher dimensions when additional regularity assumptions on $f$ are imposed. It remains open whether uniform estimates hold on $\mathbb{D}^n$ with $n \geq 2$ when $f$ is only assumed to be bounded. See also \cite{F-L-Zh11, GSS19, DPZ} for related results.
\bigskip

A class of pseudoconvex domains in $\mathbb{C}^n$ including $\mathbb D^n$ and the unit ball ${\mathbb B}^n$ are the so-called bounded symmetric domains, which up to biholomorphism are Cartesian product(s) of the Cartan classical domains of types I to IV and two domains of  exceptional types. In \cite[p. 200]{HL84}, Henkin and Leiterer asked whether there exists uniform estimate for the $\bar\partial$-equation on the Cartan classical domains of rank at least two.  Additionally, Sergeev \cite{Se94} conjectured that the $\dbar$-equation cannot be solved with uniform estimates on the Cartan classical domains of type IV of dimension $n\ge 3$.
\bigskip

Let $g = (g_{j\bar{k}})_{j, k = 1}^n$ be the Bergman metric on a domain $\Omega$.  For a $(0, 1)$-form $f=\sum_{j=1}^n f_j d\zbar_j$, one  defines
$$
 \|f\|^2_{g, \infty} :=  \esssup \big\{\sum_{j, k=1}^n g^{j\kbar} (z) f_k (z) \overline{f_j (z)}: z\in \Omega \big\},
$$
where $(g^{j\bar{k}})^{\tau} = (g_{j\bar{k}})^{-1}$ (see \eqref{norm} for details). Berndtsson used weighted $L^2$ estimates of Donnelly-Fefferman type to prove the following pointwise and uniform estimates.
 
 \begin{theorem} [\cite{Be97, Be01}]  
There is a numerical constant $C$ such that for any $\bar\partial$-closed $(0, 1)$-form $f$ on ${\mathbb B}^n$, the canonical solution to $\dbar u=f$ satisfies
 \begin{equation}\label{pointwise}
 |u(z)|\le C\|f\|_{g, \infty}\log {2\over 1-|z|},
 \end{equation} 
 and for any $\epsilon>0$,
 \begin{equation} \label{uniform} 
 \|u\|_\infty \le {C \over \epsilon} \|(1-|z|^2)^{-\epsilon}  f \|_{g, \infty}.
 \end{equation} 
 \end{theorem}
 
The estimate \eqref{pointwise} is sharp. If $f(z) := \sum_{k=1}^n z_k(|z|^2 - 1)^{-1} d\bar{z}_k$ then $f$ is $\dbar$-closed, $\|f\|_{g, \infty} = 1$ and the canonical solution to $\dbar u = f$ is $u= \log(1 - |z|^2) - C_n$, 
which shows the sharpness of  \eqref{pointwise}.
Berndtsson \cite{Be01} also pointed out his proof should generalize to other domains when enough information about the Bergman kernel is known.  Berndtsson's result \cite{Be97, Be01} was improved by Schuster and Varolin in \cite{SV14} via the ``double twisting" method. 
\medskip

Motivated by Berndtsson's results \eqref{pointwise} and \eqref{uniform} and the problems raised by Henkin and Leiterer \cite{HL84} and Sergeev \cite{Se94}, in this paper we study sharp pointwise estimates for $\dbar u = f$ for any $\dbar$-closed $(0, 1)$-form $f$ with $\|f\|_{g, \infty}<\infty$ and uniform estimates under stronger conditions on $f$.  We generalize Berndtsson's results from ${\mathbb B}^n$ to smoothly bounded strictly pseudoconvex domains and the Cartan classical domains. Our main theorem (Theorem \ref {main}/\ref{key}) for pointwise estimates is stated as follows.
 
\begin{theorem} \label{main} Let $\Omega$ be a smoothly bounded strictly convex domain, a Cartan classical domain, or the polydisc, whose Bergman kernel and metric are denoted by $K$ and $g$, respectively. Then there is a constant $C$ such that for any $\bar{\partial}$-closed $(0,1)$-form $f$ with $\|f\|_{g, \infty}<\infty$, the canonical solution to $\dbar u=f$ satisfies
\begin{equation}\label{pointwise!} 
|u(z)|\le C\|f\|_{g, \infty} \int_\Omega |K(z,w)|  dv_w, \quad z\in \Omega.
\end{equation}
\end{theorem}

\noindent{}{\bf Remarks.}

\begin{enumerate}[label=(\roman*)]

\item When $\Omega$ is a smoothly bounded strictly pseudoconvex domain, by Fefferman's asymptotic expansion for the Bergman kernel 
$$
\int_\Omega |K(z,w)| dv_w \approx  C \log { 1\over \delta_\Omega(z)} \approx \log K(z,z) ,\quad z\to \d \Omega. 
$$

In this case, the estimate (\ref{pointwise!}) is sharp. Take for example $\Omega=\mathbb{B}^n$, $u(z)=\log K(z,z)-c$ where $c$ is chosen so that $P[u]=0$.

\item  We will show in Section 3, Theorem \ref{pt-spseudo},  that  if $\Omega$ is a smoothly bounded strictly pseudoconvex domain, then \eqref{pointwise!} holds for a solution $u$ which may not be canonical.

\item When $\Omega$ is the unit polydisc $\mathbb{D}^n$, one has
$$
 \int_\Omega |K(z,w)| dv_w
 \approx  \prod_{j=1}^n  \log {2 \over 1-|z_j|}  ,\quad z\to \d \Omega.
$$

\item When $\Omega$ is a Cartan classical domain of rank greater than or equal to 2, the blow up order of  $\int_\Omega |K(z,w)| dv_w$ depends on the direction in which $z$ approaches $\partial \Omega$ and it may be larger than $-\log \delta_\Omega(z)$. For example, if $z=t I_2 \in \Omega:= $ II$(2)$, then $\int_\Omega| K(z, w)| dv_w\approx \delta_\Omega(z)^{-1 \over 2}$ as $t \to 1^-$. Moreover, if $z=t e_1\oplus t e_2\in $ IV$(n) $ with $e_j\in {\mathcal U}$ 
and $n\ge 3$
then  $\int_\Omega| K(z, w)| dv_w\approx \delta_\Omega(z)^{1 -{n\over 2}}$ as $t \to 1^-$. Here $\mathcal U$ denotes the characteristic boundary of $\Omega$.

\item In Section 6.2, we show the estimate \eqref{pointwise!} is sharp on the  Cartan classical domains.
\end{enumerate}

 Our main theorem for uniform estimates is stated as follows, as a combination of Theorems \ref{poly} and \ref{uni-spsi}.
 
 \begin{theorem} \label{main!} Let $\Omega$ be either the unit polydisc or a smoothly bounded strictly convex domain, whose Bergman kernel and metric are denoted by $K$ and $g$, respectively. Then for any $p \in (1, \infty)$, there is a  constant $C$ such that for any $\bar{\partial}$-closed $(0,1)$-form $f$,  the canonical  solution to  $\dbar u=f$ satisfies
$$
\|u\|_\infty \le C\Big\| f(\cdot )\big(\int_\Omega |K(\cdot, w)|dv_w \big)^p \Big \|_{g, \infty}.
$$
\end{theorem}

For Cartan classical domains, we give a uniform estimate under  condition \eqref{RHS} in Theorem \ref{uniform classical}. 
\bigskip

This paper is organized as follows. In Section 2, we recall and prove some properties of the Bergman kernel and metric which will be used later. In Section 3, we use $L^2$-methods to establish pointwise estimates on strictly pseudoconvex domains and the Cartan classical domains. In Sections 4 and 5, we obtain uniform estimates on the polydiscs, strictly pseudoconvex domains and the Cartan classical domains under various conditions on $f$.  In Section 6, we verify the sharpness of our pointwise estimates on the Cartan classical domains; in particular, on IV$(n)$ with $n\ge 3$ we show the estimate  has maximum  blow up order $\delta^{1-\frac{n}{2}}(z)$.

\section{Bergman kernel and metric} 

The Bergman space $A^2(\Omega)$ on a domain $\Omega \subset \mathbb C^n$ is the closed holomorphic subspace of $L^2(\Omega)$. The Bergman projection is the orthogonal projection $P_{\Omega}: L^2(\Omega) \to A^2(\Omega)$ given by 
$$
P_{\Omega}[f] (z)=\int_\Omega K (z,w) f(w) dv(w),
$$
where $K(z, w)$ is the Bergman kernel on $\Omega$ and $dv$ is the Lebesgue $\mathbb{R}^{2n}$ measure.  We will write $K(z)$ to denote the on diagonal Bergman kernel $K(z,z)$.
 When $\Omega$ is bounded, the complex Hessian of $\log K(z)$ induces the Bergman metric $B_\Omega (z; X)$ defined by
$$ 
B_{\Omega}(z; X) := \left(\sum_{j,k=1}^n g_{j\kbar} X_j\Xbar_k \right)^{1 \over 2} ,\quad g_{j\kbar}(z):= \frac{\partial^2}{\partial z_j \partial \zbar_k } \log K (z), \quad \text{for} \, \, z \in \Omega,\ X \in \mathbb C^n.
$$
The Bergman distance between $z, w \in \Omega$ is
$$
\beta_\Omega(z, w) := \inf \left\{\int_0^1 B_\Omega(\gamma(t); \gamma'(t)) dt\right\},
$$
where the infimum is taken over all piecewise $C^1$-curves $\gamma: [0,1] \to \Omega$ such that $\gamma(0) = z, \gamma(1) = w$.  Throughout the paper, 
\begin{equation} \label{hyperbolic}
B_a(\epsilon) := \{z\in \Omega: \beta_{\Omega} (z,a) \leq \epsilon\}
\end{equation}
will denote the hyperbolic ball in the Bergman metric centered at $a \in \Omega$ of radius $\epsilon$.  Additionally, $K(z, w), P_{\Omega}$ and $g$ will always denote the Bergman kernel, Bergman projection on $\Omega$ and the Bergman metric respectively.
\newline
\par

Consider a convex domain $\Omega$ that contains  no complex lines and $a \in \Omega$.  Choose any $a^{1} \in \partial \Omega$ such that $\tau_1(a) := |a - a^1| = \hbox{dist}(a, \partial \Omega)$ and define $V_1 = a + \hbox{span}(a^1 - a)^{\perp}$.  Let $\Omega_1 = \Omega \cap V_1$ and choose any $a^2 \in \partial \Omega_1$ such that $\tau_2(a) := ||a - a^2|| = \hbox{dist}(a, \partial \Omega_1)$.  Let $V_2 = a + \hbox{span}(a^1 - a, a^2 - a)^{\perp}$ and $\Omega_2 = \Omega \cap V_2.$ Repeat this process to obtain $a^1,..., a^n, w_k = {a^k - a \over ||a^k - a||}, 1 \leq k \leq n$. Define
\begin{equation} \label{Embedded Polydisk}
D(a; w, r) = \{z \in \CC^n: |\langle z - a, w_k \rangle | < r_k, 1 \leq k \leq n\}
\end{equation}
and
$$
D(a, r)= \{z \in \CC^n: |z_k - a_k| < r_k, 1 \leq k \leq n\}.
$$
By \cite[Theorem 2]{NPf03}, for convex domains that contain no complex lines, the Kobayashi metric and the Bergman metric are comparable. It follows by \cite[Corollary 2]{NT15}  that if $\Omega$ is a convex domain with no complex lines, then for every $\epsilon > 0$
there exists constants $C_1$ and $C_2$ such that for any $a$,
\begin{equation} \label{Comparable Balls} 
D(a; w, C_1\tau(a))\subset B_a(\epsilon) \subset D (a; w, C_2\tau(a)).
\end{equation}
By \cite[Theorem 1]{NPf03},
$$ 
{1 \over 4^n} \leq K(a)\prod_{j=1}^n \pi \tau_j^2(a) \leq {(2n)! \over 2^n},
$$
which implies that
$$
\left ( {C_1 \over 2} \right )^{2n} \leq K(a)v(B_a(\epsilon)) \leq (2n)! \left ( {C_2^2 \over 2} \right )^n.
$$

For any  open subset $A$ of $\Omega$, we define
$$
\|\dbar u\|_{g, \infty, A}=\big{\|} |\dbar u(z)|_{g} \big{\|}_{L^\infty (A)}.
$$
In the proofs of this paper, $C$ will denote a numerical constant, which may be different at each appearance. 
   The Cauchy--Pompeiu formula gives the following useful proposition.

\begin{pro} \label{Cauchy--Pompeiu}
Let $\Omega$ be a bounded convex domain.  For any $\epsilon > 0$ sufficiently small, there exists a constant $C$ so that for any complex-valued $C^1$ function $u$ on $\Omega$,
$$
|u(a)| \leq C \oint_{B_a(\epsilon)}|u(z)|dv_z + C\|\dbar u\|_{g, \infty, B_a(\epsilon)}.
$$

\end{pro} 

\begin{proof} 
After a complex rotation, without loss of generality, using the notation of (\ref{Embedded Polydisk}), we may assume the standard basis for $\mathbb{C}^n$ is  $(w_k)_{k=1}^n$.  Let $r_k(a) = C_1\tau_k(a)$, where $C_1$ is the same constant  as in \eqref{Comparable Balls}. 
Define the metric
$$
M_{A}(z; X) = \left({\sum_{k=1}^n {|X_k|^2 \over \tau_{k}(z)^2}}\right)^{1 \over 2}, \quad  X \in \mathbb{C}^n.
$$
It was proved in \cite {Mc3} (see also \cite{Mc, NPf03}) that
$$
M_{A}(z; X) \approx B_{\Omega}(z; X), \quad X \in \mathbb{C}^n,
$$
where $\approx$ is independent of $z$ and $X$.
So we can choose holomorphic coordinates
such that 
$$
{1\over C} D\left [{1\over \tau^2_1}, \cdots, {1\over \tau^2_n} \right]
\le \Big[g_{i\bar{j}}\Big] \le C D\left [{1\over \tau^2_1}, \cdots, {1\over \tau^2_n} \right],
$$
where $D[a_1,\cdots, a_n]$ is a diagonal matrix with diagonal entries $a_1,\cdots, a_n$.
Therefore,
$$
{1\over C} D[\tau_1^2, \cdots, \tau_n^2]
\le [g_{i\bar{j}}]^{-1}  \le C D[\tau_1^2, \cdots,  \tau_n^2].
$$
Additionally, by \eqref{Comparable Balls}  and the definition of the hyperbolic ball \eqref{hyperbolic}, it holds that
 $$\tau_k(a) \leq C \tau_k(z), \quad z \in D(a; w, C_1\tau(a)),$$ and the constant $C$ is independent of $a$.
Therefore, for $a = (a_1, ..., a_n) \in \Omega$,
\begin{align*}
r_1 (a)\|\dbar_1 u(\cdot, a_2, \ldots, a_n)\|_{L^{\infty}(D(a_1, r_1))} 
&\leq 
C\sup_{z \in D(a; w, r)} \sum_{k=1}^n r_k(z)\left| {\partial u \over \partial \bar{w}_k}(z) \right|
\\
&\leq 
C\sup_{z \in D(a; w, r)} \Big(\sum_{k=1}^n \tau_k^2(z)\left| {\partial u \over \partial \bar{w}_k}(z) \right|^2\Big)^{1/2}
\\
&\leq 
C\| \dbar u\|_{g, \infty, D(a; w, r)}
\\
&\leq  C\| \dbar u\|_{g, \infty, B_a(\epsilon)}.
\end{align*}
By Stokes' theorem, for $0<s_k <  r_k$,
$$
u(a)={1\over 2\pi i}\int_{|w_1-a_1|=s_1} {u(w_1, a_2, \cdots, a_n) \over w_1-a_1} d w_1 
+{1\over 2\pi i}\int_{|w_1-a_1|<s_1}{\d u\over \d \wbar_1}{1\over w_1-a_1} dw_1\wedge d\bar{w}_1.
$$
 By polar coordinates and (\ref{Comparable Balls}), we know that

\begin{align*}
|u(a)|
&\le {1\over \pi r_1^2} \int _{|z_1-a_1|<r_1} |u(w_1, a_2, \cdots, a_n)| dv_{z_1}+{2 r_1 \over 3} \|\dbar_1 u(\cdot, a_2,\cdots, a_n)\|_{L^{\infty}(D(a_1, r_1))} \\
&\le {1\over \pi r_1^2} \int _{|w_1-a_1|<r_1} |u(w_1, a_2, \cdots, a_n)| dv_{w_1}+C\|\dbar u\|_{g, \infty, B_a(\epsilon)}.
\end{align*}
Using the same estimate on the disc $|w_k-a_k| < s_k$ for $2 \leq k \leq n$, one gets that
\begin{eqnarray*}
|u(a)| &\leq & \oint_{D(a, C_1\tau(a)) }|u(w_1, ..., w_n)| dv_w +C\|\dbar u\|_{g, \infty, B_a(\epsilon)}   \\
&\leq & C^{2n} \oint_{B_a(\epsilon)} |u(w)| dv_w  + C\|\dbar u\|_{g, \infty, B_a(\epsilon)}.  
\end{eqnarray*}
Therefore the proof is complete.
\end{proof}

We remark that Proposition \ref{Cauchy--Pompeiu} also holds for smoothly bounded strictly pseudoconvex domains.
\newline
\par
For positive real-valued functions $f$ and $g$ on $\Omega$, we say $f \approx g$ for $z \in B_a(\epsilon)$ if for every $\epsilon > 0$ sufficiently small, there exists a constant $C = C(\epsilon, \Omega)$ so that
$$
C^{-1} \leq f(z)g(z)^{-1} \leq C, \quad z \in B_a(\epsilon).
$$
for all $a \in \Omega$.  A similar definition holds for $f \approx g$ for $z \in \Omega$.
\newline
\par
A domain $\Omega$ is homogeneous if it has a transitive (holomorphic) automorphism group. For convex homogeneous domains, the following results are known (see \cite{IY11}).

\begin{pro} \label{similar} Let $\Omega$ be a bounded homogeneous convex domain. Then,
$$
|K(z, a)| \approx K(a) \approx  {1 \over v(B_a(\epsilon))}, \quad z\in B_a(\epsilon),
$$
and for any $\epsilon > 0$, there is a $C_\epsilon$ such that  for any $a \in \Omega$,
$$
\max_{w \in B_{a}(\epsilon)}\Big{|}{K(z, w) \over K(z, a)}\Big{|} \leq C_\epsilon, \quad z \in \Omega.
$$
 
\end{pro}  

Let $\Omega $ be a smoothly bounded strictly pseudoconvex domain in $\CC^n$ and let $-r\in C^\infty(\Omegabar)$ be a strictly plurisubharmonic defining
function for $\Omega$. Define
$$
X(z,w)=r(w)+\sum_{j=1}^n {\d r(w)\over \d z_j} (z_j-w_j)+{1\over 2} \sum_{i,j=1}^n {\d^2 r(w)\over \d z_i\d z_j }(z_i-w_i)(z_j-w_j).
$$
It was proved by Fefferman \cite{Fef}, there is a $\delta>0$ such that
$$
K(z,w)={F(z,w)\over X(z, w)^{n+1} }+G(z,w) \log X(z,w)
$$
for all $(z, w)\in R_\delta (\Omega)=\{ (z,w)\in \Omega \times \Omega: r(z)+r(w)+|z-w|^2<\delta \} $, where
$F, G\in C^\infty(\Omegabar\times \Omegabar)$ and $F(z,z)>0$ on $\d \Omega$.
 \smallskip
 
When $\Omega$ is a smoothly bounded strictly convex domain. The definition of $X(z,w)$ can be simplified. In fact, one can take 
$$
X(z, a)=h_a(z) =   r(a) + \sum_{j=1}^n {\partial r \over \partial z_j}(a)(z_j - a_j).
$$
We can take $-r(z)$ to be strictly convex. Then by Taylor's theorem, one can easily see that 
$$
\Real h_a(z) \approx \Real X(z,w).
$$
Moreover,  for any $a, z\in \Omega$, write $\tilde{z}= (x_j)_{j=1}^{2n}$ and $\tilde{a} = (a_j)_{j=1}^{2n}$ if $z = (x_{2j - 1} + ix_{2j})_{j=1}^{n}$ and $a = (a_{2j - 1} + ia_{2j})_{j=1}^{n}$.
 If we apply Taylor's theorem
on the line segment between $\tilde{a}$ and $\tilde{z}$, then there is a $\theta\in (0,1)$ such that
$$
\Real h_a(z)=r(z)-\sum_{i,j=1}^{2n} {\d^2 r(\tilde{a}+\theta ( \tilde{z}-\tilde{a}) )\over \d \tilde{z}_ i\d \tilde{z}_j } (\tilde{z}_i-\tilde{a}_i)(\tilde{z}_j-\tilde{a}_j)
\approx r(z)+|z-a|^2.
$$
 Therefore, for any $a \in \Omega$,
$$
|h_a(z)| \approx  {r(z)+r(a)\over 2} + |z-a|^2+ \Big{|}\hbox{Im}  \sum_{j=1}^n {\d r(a)\over \d z_j}(z_j-a_j) \Big{|},\quad z\in \Omega.
$$
In particular this implies $h_a(z) \neq 0$. Therefore, by Fefferman's asymptotic expansion \cite{Fef} on strictly pseudoconvex domains mentioned above, we know that

\begin{lem} [\cite{Fef}] \label{haz integrals}
Let $\Omega$ be a smoothly bounded strictly convex domain.  Then, 
$$
|h_a(z)|^{-n-1} \approx K(a) \approx  {1 \over v(B_a(\epsilon))}, \quad z\in B_a(\epsilon),
$$
and there is a constant $C$ so that
$$
 \int_{\Omega} |h_a(z)|^{-n-1} dv_z \approx \int_{\Omega}  |K(z, a)| dv_z \approx \log {C\over \delta_\Omega (a)}, \qquad a \in \Omega,
$$
where $\delta_{\Omega}(\cdot)$ is the Euclidean distance function to $\partial \Omega$.  
Moreover, for any $\epsilon > 0$, there is a constant $C_\epsilon$ such that  for any $ a \in \Omega$,
$$
\max_{w \in B_{a}(\epsilon)}\Big{|}{K(z, w) h_a(z)^{n + 1}}\Big{|} \leq C_\epsilon, \quad z \in \Omega.
$$
\end{lem}

\medskip

\noindent{} {\bf Note: } We provide some insight into the integration of $|K(z,w)|$ as Forelli-Rudin type integral. Roughly,  one can view $\d \Omega$ as a space of homogeneous type with Borel measure $dX$ and quasi-distance $|X(z,w)|$.  Write
$$
|K(z,w)|\approx (\delta(z)+t +|X(\pi(z), \pi(w))|)^{-n-1},
$$
where $\pi(z)$ and $\pi(w)$ are the projections of $z$ and $w$ on $\d \Omega$ along the outer normal direction and $z,w\in R_\delta$. It follows that
$$
\int_{\d \Omega} (\delta(z)+t +|X|)^{-n-1} dX
\approx (\delta +t )^{-1}.
$$
Consequently,
$$
\int_\Omega |K(z,w)| dv(w) \approx \int_0^C (\delta(z) +t )^{-1} dt \approx \log {1\over \delta(z)}.
$$
For more information, one can consult the paper of Beatrous and the second author \cite{BL} and the papers of Krantz and the second author \cite{KL1, KL2}.

 \begin{lem} \label{finite}
Let $\Omega$ be either a smoothly bounded strictly pseudoconvex domain or a Cartan classical domain. Let $\phi(z) := \gamma\log K(z)$, $\gamma>0$. Then, for $\gamma$ sufficiently small,
$$
\int_{\Omega} e^{\phi(z)}dv_z < \infty \quad \hbox{and} \quad \|\dbar \phi\|^2_{i\d \dbar\phi}\le {1 \over 2}.   
$$
\end{lem}

\begin{proof} 
When $\Omega$ is a smoothly bounded strictly pseudoconvex domain, from Fefferman's asymptotic expansion for the Bergman kernel,
one has 
$$
\phi(z)\approx \gamma \log {1\over \delta(z)}
$$
and
$$
\int_\Omega e^{\phi(z)} dv(z) \approx \int_\Omega ({1\over \delta(z)})^\gamma dv
\approx \int_0^1 t^{-\gamma} dt <\infty.
$$
Notice that
$$
\|\dbar \phi\|^2_{i\d\dbar\phi} =\gamma \| \dbar \log K\|_g^2,
$$
where $g$ is the Bergman metric.
From Fefferman's asymptotic expansion formula (see also \cite{D94}), one gets the boundedness of  $\|\dbar \log K\|^2_g$. Choose $\gamma>0$ small
enough so that $\|\dbar \phi \|_g^2<1/2$.  
For the Cartan classical domains, the first inequality follows from  explicit formulas of the Bergman kernel \cite{Hua63}, and we compute the precise value of
$\|\partial \phi\|_{i\d\dbar \phi}$ in Section 6.1. 
\end{proof}

\section{Pointwise estimates}

An upper semicontinuous function $\phi$ defined on a domain $\Omega \subset \mathbb C^n$ with values in $\mathbb R\cup\{-\infty\}$ is called plurisubharmonic if its restriction to every complex line is subharmonic. Let $L^2(\Omega, \phi )$ denote the set of measurable functions $h$ satisfying
$ 
\int_{\Omega} |h(z)|^2e^{-\phi(z)} dv_z < \infty.
$
A $C^2$ function $\phi$ is called strongly plurisubharmonic if $i \d \dbar \phi$ is strictly positive definite. Now, let $\Omega $ be a bounded pseudoconvex domain and $\phi$ be strongly plurisubharmonic on $\Omega$. Then, for any $(0, 1)$-form $f= \sum_{k = 1}^n f_k(z) d\overline{z}_k$, define the norm of $f$ induced by $i \d \dbar \phi$ as (see also \cite{Bl05})
\begin{equation} \label{norm} 
|f|^2_{i\partial \bar{\partial}\phi} (z) := \sum_{j, k = 1}^n \phi^{j\bar{k}}(z) \overline{f_j(z)} f_k(z),
\end{equation}
where $(\phi^{j\bar{k}})^{\tau}$ equals the inverse of the complex Hessian matrix $H(\phi)$.  
Demailly's reformulation \cite{Dem82, Dem12} of H\"{o}rmander's theorem \cite{Hor65} says that {\it for any $\bar\partial$-closed $(0, 1)$-form $f$, the canonical solution in $L^2(\Omega, \phi)$
of  $\bar{\partial}u=f$ satisfies
\begin{equation} \label{Hor} 
\int_\Omega |u|^2 e^{-\phi}\, dv \leq  \int_\Omega |f|_{i\partial{\bar\partial}\phi}^2  e^{-\phi }\, dv.
\end{equation} 
}From this we see that when the $(0, 1)$-form $f$ is bounded in the Bergman metric $g$, then the canonical solution $u$ to  $\bar{\partial}u=f$ exists and satisfies the estimate \eqref{Hor}, whose right hand side is finite because it is dominated by a constant times a positive power of the Euclidean volume.
\medskip

  Donnelly and Fefferman \cite{DF83} (see also the papers by Berndtsson \cite{Be93, Be96, Be97}, McNeal \cite{Mc2} and Siu \cite{Si}) modified H\"ormander's theorem further as follows.

\begin{theorem}  [Donnelly-Fefferman type estimate] \label{DF} Let $\Omega$ be a bounded pseudoconvex domain in $\mathbb C^n$. Let $\psi$ and $\phi$ be plurisubharmonic functions on $\Omega$ such that $ i\partial \bar{\partial} \phi > 0$ and $|\d \phi|^2_{i\d \dbar\phi}\le {1 \over 2}$. Then, the canonical solution $u_0\in L^2(\Omega, \psi+{\phi \over 2})$ to $\bar{\partial}u=f$ satisfies
\begin{equation} \label{D-F} 
\int_\Omega |u_0|^2 e^{-\psi} dv\le 4 \int_\Omega |f|^2_{i\partial\bar{\partial}\phi} e^{-\psi} dv.
\end{equation}
 \end{theorem}

Next, we prove the following lemma, using the estimates \eqref{Hor} and \eqref{D-F}.

\begin{lem} \label{L2 lemma}
Let $\Omega$ be a bounded pseudoconvex domain and $f$ be a $\dbar$-closed $(0, 1)$-form on $\Omega$. Let $\psi$ and $\phi$ be plurisubharmonic on $\Omega$ and $u_0$ and $u_1$ be the $L^2$-minimal solutions to $\dbar u = f$ in $L^2(\Omega, \psi + {\phi \over 2})$ and $L^2(\Omega, \phi)$, respectively.  Suppose $B$ is a compact subset of $\Omega$ and $h \in L^{\infty}(\Omega)$ with support in $B$.

\vspace*{0.5em}
\noindent \textit{(i)} If $ i\partial \bar{\partial} \phi > 0$ and $\|\d \phi \|^2_{i\d \dbar\phi}\le {1 \over 2}$ on $\Omega$, then
\begin{equation} \label{L2DF1}
\int_{B}  |u_0| dv \le 2 \Big(\int_\Omega |f|^2_{i\d \dbar \phi} e^{-\psi} dv \Big)^{{1 \over 2}} \Big(\int_B  e^{\psi} dv \Big)^{{1 \over 2}}
\end{equation}
and
\begin{equation} \label{L2DF2}
\Big|\int_{\Omega}  u_0 \overline{P(h)} dv \Big |\le 2 v(B) \|h\|_{\infty} \Big(\int_\Omega |f|^2_{i\d \dbar \phi} e^{-\psi} dv \Big)^{{1 \over 2}} \Big(\int_\Omega \max_{w\in B} |K(z,w)|^2  e^{\psi(z)} dv_z \Big)^{{1 \over 2}};
\end{equation}

\vspace*{0.5em}
\noindent \textit{(ii)}
$$
\int_{B}  |u_1| dv  \le 2 \Big(\int_\Omega |f|^2_{i\d \dbar \phi} e^{-\phi} dv\Big)^{{1 \over 2}} \Big(\int_B  e^{\phi} dv \Big)^{{1 \over 2}}
$$
and
$$
\Big|\int_{\Omega}  u_1 \overline{P(h)} dv \Big |  \le 2  v(B)\|h\|_{\infty} \Big(\int_\Omega |f|^2_{i\d \dbar \phi} e^{-\phi} dv \Big)^{{1 \over 2}} \Big(\int_\Omega \max_{w\in B} |K(z,w)|^2  e^{\phi(z)} dv_z \Big)^{{1 \over 2}}.
$$
\end{lem}

\begin{proof} Let $\chi_{B}$ denote  the characteristic function on $B$, and let $\beta:=\chi_{B} {u_0(z) \over |u_0(z)|} $.  By \eqref{D-F},
$$
\left(\int_{B} |u_0| dv \right)^2=\left|\int _{\Omega}u_0 \,\bar{\beta} \, dv\,\right| ^2 \leq  \int\limits_\Omega |u_0 |^2e^{-\psi} dv \cdot \int\limits_{B}|\beta|^2e^{ \psi} dv 
 \leq  4   \int\limits_\Omega |f|^2_{i \partial \bar{\partial}\phi}e^{-\psi} dv  \int\limits_{B}   e^{\psi} dv,
$$
which proves \eqref{L2DF1}.  Notice that
\begin{align*}
\left|\int _{\Omega}u_0\overline{P(h )} dv \right|^2
&\leq  \int_{\Omega} |u_0|^2 e^{-\psi} dv \cdot  \int_\Omega |P(h) |^2 e^{\psi} dv \\
&\leq  4 \int_\Omega |f |_{i \d\dbar \phi}^2 e^{-\psi } dv \cdot v^2(B)  \cdot \|h\|^2_\infty \cdot  \int_\Omega \max_{w\in B} |K(z,w)|^2 e^{\psi(z)} dv_z,
\end{align*}
which proves \eqref{L2DF2}. Part (ii) can be proved similarly using H\"ormander's estimate \eqref{Hor} in place of Donnelly-Fefferman's estimate \eqref{D-F}. 
\end{proof}

\begin{theorem} [Key Estimate] \label{key} Let $\Omega$ be a Cartan classical domain or a smoothly bounded strictly convex domain.  Then, there is a constant $C$ such that for any $\bar{\partial}$-closed $(0,1)$-form $f$ on $\Omega$ with $\|f\|_{g, \infty}<\infty$, the canonical solution to $\dbar u = f$ satisfies
$$
|u(z)|\le C\|f\|_{g, \infty} \int_\Omega |K(z,w)|  dv_w, \quad z \in \Omega.
$$
\end{theorem}

\begin{proof}
From the discussion after \eqref{Hor} we see that the canonical solution to  $\bar{\partial}u=f$ exists.
Suppose first $\Omega$ is a Cartan classical domain.  For an arbitrary $a\in \Omega$ and any sufficiently small $\epsilon>0$, let $\beta:=\chi_{B_a(\epsilon)} {u(z) \over |u(z)|} $, where $\chi_{B_a(\epsilon)}$ is the characteristic function of the hyperbolic ball $B_a(\epsilon)$. Let $\phi:=\gamma \log K(z)$ be a plurisubharmonic function on $\Omega$ for some chosen $\gamma$ that satisfies the condition in Lemma \ref{finite}. Define $\psi(z)=:\psi_{a} (z):=-\log |K(z, a)|$. Then $\psi_{a} $ is pluriharmonic and bounded on $\Omega$. 
Also define the function
$$
\phi_0:=\psi_{a} +{\phi\over 2},
$$
and let $u_0$ be the $L^2(\Omega, \phi_0)$ minimal solution to the equation $\bar{\partial}v=f$. Then by Theorem \ref{DF},
$$
\int_\Omega |u_0|^2 e^{-\psi} dv \le 4\gamma^{-1} \int_\Omega |f|^2_{ g} e^{-\psi} dv \le  4\gamma^{-1} \|f\|^2_{g, \infty} \int_\Omega   e^{-\psi} dv < \infty,
 $$
which implies that $ u_0 \in L^2(\Omega)$.
So
$u - u_0 \in A^2(\Omega)$ and
$$
\int _{B_a(\epsilon)}|u| dv=\int _\Omega u\bar{\beta} dv =\int \limits_\Omega u(\overline{\beta-P (\beta)}) dv =\int_\Omega u_0( \overline{\beta-P(\beta)}) dv
=\int_\Omega u_0 \overline{\beta} dv-\int_\Omega u_0 \overline{P (\beta)} dv.
$$
By Lemma \ref{finite} and \eqref{L2DF1} in Lemma \ref{L2 lemma},
\begin{align*}
\left|\int _{\Omega}u_0 \,\bar{\beta} \, dv\,\right|^2
& \leq 4  \int\limits_\Omega |f |^2_{i\partial \bar{\partial}\phi}e^{-\psi_{a} } dv  \cdot  \int\limits_{B_a(\epsilon)}   e^{\psi_{a}} dv  \\
&\leq C \|f\|^2_{g, \infty}   \int_{\Omega} |K(z,a)| dv_z  \cdot \int_{B_a(\epsilon)} |K(z,a)|^{-1} dv_z  \\
&\leq C_{\epsilon} \|f\|^2_{g, \infty}   \int_{\Omega} |K(z,a)| dv_z \cdot v(B_a(\epsilon)) \cdot K(a)^{-1} \\
&\leq C_{\epsilon} \|f\|^2_{g, \infty} v^2(B_a(\epsilon)) \int_{\Omega} |K(z,a)| dv_z,
\end{align*}
where the last two inequalities hold due to Proposition \ref{similar}, and $C_\epsilon$ is a constant depending on $\epsilon$. On the other hand, by \eqref{L2DF2} in Lemma \ref{L2 lemma} and Proposition \ref{similar} again,
\begin{align*}
 \left | \int _{\Omega}u_0 \overline{P(\beta)} dv \right |^2 &\leq 
 C v^2(B_a(\epsilon))  \int_\Omega |f |_{i\d\dbar \phi}^2 e^{-\psi_{a}} dv \cdot   \int_\Omega \max_{w\in \overline{B_a(\epsilon)}}|K(z, w)|^2 e^{\psi_{a}(z)} dv_z \\ 
 &\leq C_{\epsilon}  v^2(B_a(\epsilon)) \int_\Omega |f |_{i\d\dbar \phi}^2(z) |K(z,a)| dv_z    \cdot  \int_{\Omega} |K(z,a)|^{2-1}  dv_z \\ 
 &\leq C_{\epsilon} \|f\|^2_{g, \infty} v^2(B_a(\epsilon))     \left(\int_{\Omega} |K(z,a)|dv_z\right)^2.
\end{align*}
Combining the above estimates, one can see easily 
$$
\frac{1}{ v(B_a(\epsilon))} \int \limits_{B_a(\epsilon)}|u| dv \le C_{\epsilon}\|f\|_{g,\infty}  \int_\Omega |K(z,a)| dv_z.
$$
Fix $\epsilon > 0$.  By Proposition \ref{Cauchy--Pompeiu}, there exists a constant $C$ depending only on $\Omega$ such that
$$
|u(a)|\le  C\|f\|_{g,\infty}  \int_{\Omega} |K(z,a)| dv_z . 
$$

If $\Omega$ is instead a smoothly bounded strictly convex domain, then let $\psi_a(z) = (n+1) \log|h_a(z)|$, repeat the argument for the Cartan classical domains and use Lemma \ref{haz integrals}.
\end{proof}

In Section 6, Proposition \ref{eg}, we show that the estimate in Theorem \ref{key} is sharp for the Cartan classical domains. When $\Omega$ is the unit ball $\mathbb{B}^n$, the Key Estimate reduces to Berndtsson's result \eqref{pointwise}. Now we generalize \eqref{pointwise} and  \eqref{pointwise!} to smoothly bounded strictly pseudoconvex domains.

\begin{theorem} \label{pt-spseudo}
Let $\Omega$ be a smoothly bounded strictly pseudoconvex domain.  Then there is a constant $C$ such that for any $\dbar$-closed $(0, 1)$-form $f$ on $\Omega$ with $\|f\|_{g, \infty} < \infty$, there is a solution $u$ to $\dbar u = f$ such that
$$
|u(z)| \le C \|f\|_{g, \infty} \log (1+K(z)), \quad z \in \Omega.
$$
\end{theorem}

\begin{proof} Let $r(z)$ be a strongly plurisubharmonic defining function for $\Omega$ such that $r(z)\in C^\infty(\Omegabar)$ and $r > 0$ in $\Omega$. Consider the polynomial
$$
X(z, w):= r(w)+  \sum_{j=1}^n \left. {\d r\over \d w_j}\right|_w (z_j-w_j)+{1\over 2}\sum_{j,k=1}^n \left. {\d^2 r \over \d w_j \d w_k} \right|_w (z_j-w_j)(z_k-w_k).
$$
Define the region $
R_\epsilon =\{(z, w) |z, w\in \Omega: r(z)+r(w)+|z-w|^2<\epsilon \}.$ For $(z,w)\in R_\epsilon$, Fefferman \cite{Fef} showed the Bergman kernel on $\Omega$ can be expressed as
\begin{equation} \label{expansion}
K(z,w)={F(z,w) \over X(z,w)^{n+1}}+ G(z,w) \log X(z,w),
\end{equation}
where $G, F\in C^\infty (\overline{\Omega} \times \overline{ \Omega})$, $F(z,z)>0 $ on $(\overline{\Omega} \times \overline{\Omega}) \cap R_\epsilon$ and $``\log"$ denotes the principal branch of the logarithm defined on $\mathbb{C} \setminus (-\infty, 0]$. The asymptotic expansion \eqref{expansion} implies that
\begin{equation} \label{int str}
\int_\Omega |K(z,w)| dv_w\le C(1+\log K(z) ),\quad z\in \Omega.
\end{equation}
Since the boundary $\d \Omega$ is compact, for any $\delta > 0$, there are finitely many $b^j\in \d \Omega, \,  j=1,..., m,\,$ such that
$
\d \Omega \subset \cup_{j=1}^m \mathbb B(b^j, \delta).
$
Choose smoothly bounded strictly pseudoconvex domains $\Omega^j, j= 1,..., m$ such that
$$
 \mathbb B(b^j, 3\delta)\cap \Omega \subset \Omega^j \subset \Omega \cap  \mathbb B(b^j, 4\delta).
$$
where $\delta$ is chosen small enough such that for each $j$, after a polynomial change of variables, each $\Omega^j$ is a strictly convex domain.  Let $\{\Omega^j\}_{j=m + 1}^{m + k}$ be a finite open cover of $\Omega \setminus \cup_{j=1}^{m}(\mathbb{B}(b^j, \delta) \cap \Omega)$ consisting of balls contained in $\Omega$.  In the argument of Theorem \ref{key} by letting $\phi_0 = \gamma \log K_{\Omega}(z)$ (instead of $\gamma \log K_{\Omega^j}(z)$) and using \eqref{int str}, we can solve the equation $\dbar u^j = f$ on $\Omega^j$ with minimal solution $u^j$ satisfying
\begin{equation} \label{u^j}
 |u^j(z)| \le C\|f\|_{g, \infty} \log (1+K (z)).
\end{equation} 

Let $\{\eta_j\}_{j=1}^{m + k}$ be a partition of unity of $\overline{\Omega}$ subordinate to the cover $\{B(b^j, \delta)\}_{j=1}^m \cup \{\Omega^j\}_{j= m + 1}^{m + k}$ and let
$v(z):=\sum_{j=1}^{m + k} \eta_j(z) u^j(z).$ Then $\dbar v=f+h,$ where $h:=\sum_{j=1}^{m + k} u^j \dbar \eta_j$ is a $\dbar$-closed (0,1)-form on $ \Omega$. By the integral formula in \cite{Hen70, G-L70}, there is a bounded solution $v_0$ to the equation $\dbar v_0=h$.  Let $u = v - v_0$.  Then 
$\dbar u = f$ and by \eqref{u^j},
$$
|u(z)| \leq  \sum_{j=1}^k \eta_j(z)C\| f\|_{g, \infty} \log(1 + K(z)) + C
\leq  C\|f\|_{g, \infty} \log(1 + K(z)).
$$
Therefore, the proof is complete. \end{proof}

\medskip

\noindent{}{\bf Remark}.

For a smoothly bounded strictly pseudoconvex domain $\Omega$, if the canonical solution is $u_0$, then for $h\in L^\infty(\Omega)$ with $\|h\|_\infty\le 1$,
$$
\Big|P[ h(\cdot) \log K(\cdot)](z) \Big|\le C\big (1+\log K(z)\big )^2.
$$
In fact, letting $\omega_t=\{z\in \Omega: \delta(z)>t\}$, by the Fefferman's expansion theorem on Bergman kernel \cite{Fef}, we know that
\begin{eqnarray*}
|P[h(\cdot) \log K (\cdot)](z)|
&\le& \int_{\Omega} |h(w)| \log K(w) |K(z,w)| dv(w)\\
&\le& \|h\|_\infty \int_{\Omega} \log K(w) |K(z,w)| dv(w)\\
&\approx& \|h\|_{\infty} \int_0^c \int_{\d \Omega_t} \log K(w) |K(z,w)| d\sigma(w) d t\\
&\le & C\|h\|_\infty \int_0^c (- \log t )  {1\over \delta(z)+ t} d t\\
&\le & C\|h\|_\infty \Big( {1\over \delta(z)} \int_0^{\delta(z)} (-\log t dt )+\int_{\delta(z)} ^c  -\log t   {1\over \delta(z)+ t} d t\Big) \\
&\le & C\|h\|_\infty \Big( \log {C\over \delta(z)} +\int_{\delta(z)} ^c  -\log (\delta(z)+ t )  {1\over \delta(z)+ t} d t\Big) \\
&\le & C\|h\|_\infty \Big( \log {C\over \delta(z)} \Big)^2.
\end{eqnarray*}
 Combining this with Theorem \ref{pt-spseudo}, one gets
$$
\big |u_0(z) \big|\le C\big (1+\log K(z)\big)^2.
$$

\section{Uniform estimates}

In this section, we obtain uniform estimates for the equation $\bar \partial u=f$ on the unit polydisc $\mathbb{D}^n$ and strictly pseudoconvex domains by imposing conditions on $f$ stronger than $\|f\|_{g, \infty} < \infty$.

\begin{theorem} \label{poly} For any $p \in  (1, \infty)$, there is a constant $C$ such that for any $\dbar$-closed $(0, 1)$ form $f$ on $\mathbb{D}^n$, the canonical solution $u$ to $\dbar u = f$ satisfies
$$
\|u\|_{\infty} \le C \Big \| f  \prod_{j=1}^n \Big ( \log \big( { 2\over 1 - |z_j|^2}\big) \Big ) ^{p} \Big\|_{g, \infty}.
$$
\end{theorem}

\begin{proof} Let $A_0 = 2p + \log v(\mathbb D^n)$.  Choose $0 < \gamma < \sfrac{1}{2}$ so that $\phi (z):= \gamma \log K(z)$ satisfies Lemma \ref{finite}. 
Choose $\alpha>1$ so that $\alpha-\gamma=1$. Let $K_j(z) := \pi^{-1} (1-|z_j|^2)^{-2}$ and let
$$
\phi_0(z) := \phi(z) - \sum_{j=1}^n p\log(A_0 + \gamma\log K_j(z)) - \alpha \log|K(z, a)|.
$$
Since on $\mathbb D^n$, $\log|K(z, a)|$ is pluriharmonic and 
$$
A_0 + \gamma \log K_j = 2p + n\log \pi - \gamma \log \pi - 2\gamma \log (1 - |z_j|^2) \geq 2p,
$$
we know that
$$
i\d\dbar \phi_0 = \sum_{j=1}^n i\d\dbar(\gamma \log K_j)\big(1 - {p \over A_0 + \gamma \log K_j}\big) + \pi {\d(\gamma \log K_j) \wedge \dbar(\gamma \log K_j) \over (A_0 + \gamma \log K_j)^2} \geq  {1 \over 2}\sum_{j=1}^n i\d\dbar (\gamma \log K_j).
$$
Thus, $|f|^2_{i\d\dbar \phi_0}\le 2 | f|^2_{i\d\dbar\phi} = {2 \over \gamma}|f|^2_g.$ Let 
$$
A_p(f) = \Big\| f \Big{(}\prod_{j=1}^n (A_0 - \gamma \log\pi - 2\gamma\log (1 - |z_j|^2)\Big{)}^p  \Big \|_{g,  \infty}
$$
As in Theorem \ref{key}, let $u_0$ be the $L^2(\mathbb D^n, \phi_0)$ minimal solution to $\bar{\partial}u_0=f$ and $\beta := e^{i\theta(z)}\chi_{B_a(\epsilon)}$ where $u(z) = |u(z)|e^{i\theta(z)}$.  Then,
\begin{align*}
\int_{B_a(\epsilon)} |u| dv \leq &  Cv(B_a(\epsilon))\Big{(}\int_{D(0, 1)^n} |f|^2_{g}e^{-\phi_0} dv\Big{)}^{1 \over 2}\Big{(}\int_{D(0, 1)^n} |K(z, a)|^2 e^{\phi_0(z)} dv_z\Big{)}^{1 \over 2}
\\
\leq  &C A_p(f) v(B_a(\epsilon)) \Big{(}\int_{D(0, 1)^n } {e^{-\phi_0} dv_z \over \prod_{j=1}^n (A_0 + \gamma \log K_{j}(z))^{2p} }   \int_{D(0, 1)^n} |K(z, a)|^2e^{\phi_0} dv_z\Big{)}^{1 \over 2}
\\
= &C A_p(f) v(B_a(\epsilon)) \prod_{j=1}^n \Big{(}\int_{D(0, 1)} {|K_{j}(z, a)|^{\alpha} K_{j}(z)^{-\gamma}  dv_z  \over ( A_0 + \gamma \log K_{j}(z))^{p}}     \int_{D(0, 1)} {|K_{j}(z, a)|^{2 - \alpha} 
K_{j}(z)^{\gamma} dv_z \over  (A_0 + \gamma \log K_{j}(z))^p} \Big{)}^{1 \over 2}
\\
\leq  &CA_p(f) v(B_a(\epsilon)).
\end{align*}
Fix sufficiently small $\epsilon>0$. By Proposition \ref{Cauchy--Pompeiu},
$$
|u(a)| \le  C A_p(f) +C\|f\|_{g, \infty} \leq
C \Big \| f  \prod_{j=1}^n \Big( 2(n + p) - \log \big ( {1 - |z_j|^2}\big ) \Big )^{p} \Big\|_{g, \infty}. 
$$
Notice that 
$$
 2(n + p) - \log \big ( {1 - |z_j|^2}\big ) \le 5(n+p) \log {2\over 1-|z_j|^2},
 $$
 which completes the proof of the theorem. 
\end{proof}

In fact, the finiteness of the right hand side of the estimate in Theorem \ref{poly} is a stronger condition than $f$ being $L^\infty$ on $\mathbb{D}^n$. 
Recently, the first author and Pan and Zhang in \cite{DPZ} obtained uniform estimates for the canonical solution to $\bar \partial u=f$  when $f$ is continuous up to the boundary of $\mathbb{D}^n$, and more generally the Cartesian product of smoothly bounded planar domains.

However, the situation for strictly pseudoconvex domains is quite different.
The finiteness of the right hand side of the estimate in either Theorem \ref{uni-spsi} 
or Theorem \ref{uni-pseu} is a much weaker condition than $f$ being $L^\infty$ on each smooth domain considered. In fact, we allowed $f$ to
blow up on $\d \Omega$ to order less than $1/2$.

\begin{theorem} \label{uni-spsi} Let $\Omega$ be a smoothly bounded strictly convex domain.  For any $p \in  (1, \infty)$ and sufficiently small $\gamma>0$,
 there exists a constant $C$ such that for any $\dbar$-closed $(0, 1)$-form $f$, the canonical solution $u$ to $\dbar u = f$ satisfies
$$
\|u\|_{\infty} \le C \| (1+\log v(\Omega)+ \gamma \log K(z))^p f\|_{g, \infty}.
$$

\end{theorem}

\begin{proof} 
Choose $0 < \gamma < \sfrac{1}{(n+2)}$ so that $\phi(z):=\gamma \log K(z)$ satisfies Lemma \ref{finite} and let $\alpha = \gamma + 1$. Let $
A_0:= 2p + \gamma \log v(\Omega)
$
and let
$$
\phi_0(z)=\phi(z)-(n+1) \alpha \log h_a(z) -p \log (A_0 + \gamma \log K(z)).
$$
Notice that
$$
i\d\dbar \phi_0 \geq \Big(1-{p \over A_0+ \phi}\Big) i \d\dbar \phi \geq \frac{i \d\dbar \phi}{2}
$$
and $\gamma K(z) > \gamma v(\Omega)^{-1}.$ Therefore 
$$
|f|^2_{i\d\dbar \phi_0} \leq {2 \over \gamma} |f|^2_{g}.
$$ Define
$$
A_p(f) := \left \| (A_0 + \gamma \log K(z))^p f \right \|_{g, \infty}.
$$
Using arguments similar to those in Theorems \ref{key} and \ref{poly} and $\alpha-\gamma=1$,
\begin{align*}
\int_{B_a(\epsilon)} |u| dv &\leq  Cv(B_a(\epsilon))\Big{(}\int_{\Omega} |f|^2_{i\partial \dbar \phi}e^{-\phi_0} dv\Big{)}^{1 \over 2}\Big{(}\int_{\Omega} |K(z, a)|^2 e^{\phi_0(z)} dv_z\Big{)}^{1 \over 2}
\\
&\leq  CA_p(f)v(B_a(\epsilon)) \Big{(}\int_{\Omega} {e^{-\phi_0} \over (A_0 + \gamma \log K(z))^{2p}} dv_z  \int_{\Omega} |K(z, a)|^2e^{\phi_0} dv_z\Big{)}^{1 \over 2}
\\
&= CA_p(f)v(B_a(\epsilon)) \Big{(}\int_{\Omega} {K(z)^{-\gamma}|h_a(z)|^{\alpha(n+1) } \over (A_0 + \gamma \log K(z))^{p} } dv_z \int_{\Omega} {|K(z, a)|^2K( z)^{\gamma}|h_a(z)|^{-\alpha (n+1)}  \over (A_0 + \gamma\log K(z))^p}  dv_z \Big{)}^{1 \over 2}
\\
&\leq  CA_p(f)v(B_a(\epsilon)) \Big{(}\int_{\Omega} {|K(z, a)| dv_z \over (A_0 + \gamma \log K(z))^{p}}  \int_{\Omega} {|K(z, a)|^{2 - \alpha} K(z)^{\gamma} \over (A_0 + \gamma \log K(z))^p} dv_z\Big{)}^{1 \over 2}
\\
&\leq  CA_p(f)v(B_a(\epsilon)),
\end{align*}
where the last inequality follows from Fefferman's asymptotic expansion. In fact, since ${n \over n + 1} < 2 - \alpha$, if 
$
\Omega_t = \{z: r(z) > t\}
$
where $r(z)$ is a defining function satisfying the definition of $h_a(z)$, then
\begin{eqnarray*}
 \int_{\Omega} {|K(z, a)|^{2 - \alpha} K(z)^{\gamma} \over (A_0 + \gamma \log K(z))^p} dv_z
& \le & C\Big [1+ \int_0^\epsilon \int_{\d \Omega_t } {|K(z, a)|^{2 - \alpha} K(z)^{\gamma} \over (A_0 + \gamma \log K(z))^p}  d\sigma_t(z) dt \Big]\\
& \le & C\Big [1+ \int_0^\epsilon {t^{-(2 - \alpha)(n+1)+n} t^{-\gamma(n+1)} \over (A_0 - \gamma (n+1) \log t)^p}  dt \Big]\\
& \le & C\Big [1+ \int_0^\epsilon {1 \over  t ( \log t)^p}  dt \Big]\\
&\le& C \Big[1-{\log \epsilon \over p-1}\Big].
\end{eqnarray*}
 By Proposition \ref{Cauchy--Pompeiu}, for a fixed $\epsilon > 0$ sufficiently small,
$
|u(a)|\le C A_p(f) +\|f\|_{g, \infty} \leq C A_p(f).
$
\end{proof}

Using an argument similar to the proof of Theorem \ref{pt-spseudo} we get the following generalization of Theorem \ref{uni-spsi} to smoothly bounded strictly pseudoconvex domains.
 
\begin{theorem} \label{uni-pseu} Let $\Omega$ be a smoothly bounded strictly pseudoconvex domain. Then, for any $p \in  (1, \infty)$, there exists a constant $C$ such that for any $\dbar$-closed $(0,1)$-form $f$, there is a solution $u$ to $\bar \partial u=f$ that satisfies
$$
\|u\|_{\infty} \le C \| (\log K(\cdot))^p f(\cdot)\|_{g, \infty}.
$$
\end{theorem}

\medskip

\noindent{\bf Remark.}

 Let $f\in L^\infty_{(0,1)}(\Omega)$ be a $\dbar$-closed form on a smoothly bounded strictly pseudoconvex domain $\Omega$. Henkin and Romanov's theorem in \cite{HR} states that there exists a solution $u\in C^{1/2}(\Omega)$. 
Theorem \ref{pt-spseudo} implies that one can find a bounded solution when
$(\log {1\over \delta(z)})^p  \delta(z) |f(z)|^2 $ is bounded. Moreover, Lieb and Range in \cite[Theorem 2 (i)]{LR} showed that uniform estimates hold for the canonical solutions to the
$ \bar \partial$-equations on $\Omega$.

\section{Additional estimates for Cartan classical domains}

A domain $\Omega$ is symmetric if for all $a \in \Omega$, there is an involutive  automorphism $G$ such that $a$ is isolated in the set of fixed points of $G$.  
All bounded symmetric domains are convex and homogeneous. E. Cartan proved that all bounded symmetric domains in $\mathbb C^N$ up to biholomorphism are the Cartesian product(s) of 
the following four types of Cartan classical domains and two domains of exceptional types. 

\begin{definition} A Cartan classical domain is a domain of one of the following types:

\begin{enumerate}[label=(\roman*)]
\item I$(m, n) := \{z \in M_{(m, n)}(\mathbb C): I_m - zz^{*} > 0\},\, m \leq n$;
\item II$(n) := \{z \in I(n, n): z^{\tau} = z\}$;
\item III$(n) := \{z \in I(n, n): z^{\tau} = - z\}$;
\item IV$(n) := \{z \in \mathbb C^n: 1 - 2|z|^2 + |s(z)|^2 > 0 \hbox{ and } |s(z)| < 1\}$, where $s(z) := \sum_{j=1}^n z_j^2$ and $n > 2$.
\end{enumerate}
Here $z^{*}:=\bar{z}^{\tau}$ is the conjugate transpose of $z$.
\end{definition}

Let $\Omega$ be a Cartan classical domain.  Denote the rank, characteristic multiplicity, genus, complex dimension and kernel index of $\Omega$  
by $r, {\bf a}, p, N$ and $k$, respectively.  Their values are given in the following chart.

\medskip

\begin{center} 
\begin{tabular}{|c|ccccc|}
\hline 
{ Classical Domain}  & { rank $r$ }  &   {multiplicity {\bf a}}  & genus $p$ & {  dimension $N$ } & index $k$\\
 \hline 
I(m, n), $m \leq n$ &  m  & 2 & m+n & mn & 1\\
\hline
II(n) &  n  & 1 & n+1 & $n(n+1)/2$  & 1\\
\hline
III(2n+$\epsilon$), $\epsilon=0$ or 1 &  n  & 4 & $2(2n+\epsilon-1)$ & $n(2n + 2\epsilon -1)$  & ${1/2}$\\
\hline
IV(n) &  2  & n-2 & $ n$ & n  & 1\\
\hline
\end{tabular}
\captionof{table}{Characteristics of Classical Domains}
\end{center}

\medskip

Hua \cite{Hua63} obtained explicit formulas for the Bergman kernels on the Cartan classical domains. For a domain $\Omega$ of type I, II or III,
$$
K(z,w)=C_{\Omega}\left[\det (I-z w^*)\right]^{-pk},
$$
 and for a domain of type IV, 
$$
K(z,w)=C_{n} [1-2 \sum_{j=1}^n z_j \wbar_j +s(z) \overline{s(w)} ]^{-n}.
$$
For any $z\in \Omega$,
$$
\delta_\Omega(z) \leq  K(z)^{-1 \over rpk}.
$$
Let $\lambda=p k$. By Theorem 3.8 in \cite{F-K90}, one can write the Bergman kernel on a Cartan classical domain $\Omega$ as follows:
$$
K(z,w)=h(z,w)^{-\lambda}=\sum_{{\bf m}\ge 0} (\lambda)_{\bf m} K_{\bf m} (z,w),
$$
where
$$
{\bf m}=(m_1,\cdots, m_r) \ \hbox{ and } \  {\bf m}\ge 0 \iff m_1\ge m_2\ge \cdots\ge m_r\ge 0
$$
and
$$
(\lambda)_{\bf m}={\Gamma_{\Omega}(\boldsymbol\lambda+{\bf m})\over \Gamma_\Omega(\boldsymbol\lambda)},\quad \Gamma_\Omega({\bf s})=c_\Omega \prod_{j=1}^r \Gamma \Big(s_j-(j-1){a\over 2}\Big), \quad \boldsymbol\lambda = (\lambda,..., \lambda).
$$
Here, $K_{\bf m}$ is the Bergman kernel for homogeneous polynomials in $\mathbb{C}^r$ of degree $|{\bf m}|=m_1+ \cdots+ m_r$.
For each Cartan domain $\Omega$, there is a subgroup $\mathcal{K}(\Omega)$ of the unitary group such that for each $z\in \Omega$ there is $k \in \mathcal{K}(\Omega)$ such that
$z=k \tilde{z}$ where $\tilde{z}\in \mathbb{C}^r \times \prod_{j = r +1}^{N} \{0\}$ and $K_{\bf m}(z,z)=: K_{\bf m}(\tilde{z}, \tilde{z})$.

The following Forelli-Rudin type integral was studied by Faraut and Koranyi in \cite{F-K90}:
$$
J_{\beta, c}(z):=\int_\Omega K(w)^{\beta} |K(w, z)|^{1+c-\beta} dv_w.
$$
By the proof of Theorem 4.1 in  \cite{F-K90}, one has
$$
J_{\beta, c}(z)=\sum_{{\bf m}\ge 0} {|(\mu)_{\bf m}|^2 \over ((1-k \beta) p)_{\bf m}} { K}_{\bf m}(z,z),\quad \mu={kp\over 2} (1+c-\beta).
$$
 Using Stirling's formula, one can show (see (4.3) in \cite{F-K90}, or (2.9) in \cite{E-Z04}) that as $m$ varies, 
$$
  {|({pk (1-\beta)\over 2})_{\bf m}|^2 \over ((1-k \beta) p)_{\bf m}} \approx   { ({k p\over 2})_{\bf m}^2 \over ( p)_{\bf m}},
$$ 
which  implies that \begin{equation}\label{J_0}
J_{\beta,0}(z)\approx J_{0,0}(z)=\int_\Omega |K(z,w)| dv(w),\quad  \beta<{1\over pk}.
\end{equation}

More computations were carried out by  Faraut and Koranyi in \cite{F-K90}.
\begin{theorem} [\cite{F-K90}] \label{Thm-F-K} For any $\beta<{1 \over pk}$,

\item \quad (i)  $J_{\beta, c}(z)$  is bounded for all $z \in \Omega$ if and only if $ c<-\frac{(r-1){\bf a}}{2p};$
   \item \quad (ii) $J_{\beta, c}(z) \approx K( z)^{c},$ \quad if $ c>\frac{(r-1){\bf a}}{2p}$.
\end{theorem}

When $|c|\le (r-1){\bf a}/(2p)$, it is difficult to compute $J_{\beta,c}(z)$ (see \cite{Kor, Yan90}).
 Theorem 1 of \cite{E-Z04}, whose parameters are chosen as ${p\over 2}(1+c-\beta), {p\over 2}(1+c-\beta)$ and $ p(1-\beta)$, is stated as follows.

\begin{theorem} \label{Thm-LZ} Let $\Omega$ be a  Cartan classical domain of rank $2$ with characteristic boundary ${\mathcal U}$. Then for any $z= t e_1+ Te_2$ with $0\le t\le T<1$ and $e_1, e_2\in {\mathcal U}$ the following statements hold:
\begin{enumerate} [label=(\roman*)]

\item if $2pc={\bf a}$, then $ J_{\beta, c}( z) \approx (1-t)^{-{{\bf a} \over 2}} (1-T)^{-{{\bf a} \over 2}} [1-\log(1-t)]$;

\item if $0<2pc<{\bf a}$, then $J_{\beta, c}\approx (1-t)^{-{{\bf a} \over 2}} (1-T)^{-pc}$;

\item if $c=0$, then $J_{\beta, c}(z)\approx (1-t)^{-{{\bf a} \over 2}} [1+\log{1-t\over 1-T}]$;

\item if $-{\bf a}<2pc<0$, then $J_{\beta, c} (z)\approx (1-t)^{-pc-{{\bf a} \over 2}}$;

\item if $2pc =-{\bf a}$, then $J_{\beta,c}(z)\approx 1- \log (1-t)$.

\end{enumerate}
\end{theorem}

As a consequence, when $\Omega$ is a Cartan classical domain of rank $2$ and $z=t e_1+ t e_2$ with $0\le t <1$ and $e_i\in {\mathcal U}$, one has
$$
\int_\Omega |K(z, w)| dv_w\approx (1-t)^{-{{\bf a} \over 2}}\approx \delta_\Omega(z)^{-{{\bf a} \over 2}}.
$$

On the Cartan classical domains, we impose a stronger
assumption on $f$ to get bounded solutions to $\bar \partial u=f$. The following result provides 
a partial answer to the problems raised by Henkin and Leiterer \cite{HL84} and Sergeev \cite{Se94}.

\begin{theorem} \label{uniform classical} Let $\Omega$ be a Cartan classical domain and  $\alpha>1+\frac{(r-1){\bf a}}{2p}$.  Then, there exists a constant $C$
 such that for any $\dbar$-closed $(0, 1)$-form $f$, the canonical solution $u$ to $\dbar u = f$ satisfies
\begin{equation} \label{RHS}
\|u\|_{\infty} \le C \Big \|   \int _{\Omega}|f|^2_g(z)  |K(z, \cdot)|^{\alpha}dv_z  \Big \|^{1 \over 2}_{\infty} + C \|f\|_{g, \infty}.
\end{equation}
\end{theorem}

\begin{proof}

As in the proof of Theorem \ref{key}, let $\beta:=\chi_{B_a(\epsilon)} {u(z) \over |u(z)|},$ $\phi:=\gamma \log K(z)$, and $\psi_{a} (z):=-\alpha \log |K(z, a)|$ for any $a \in \Omega$. Then,
$$
\int_{B_a(\epsilon)}|u| dv \leq \int_\Omega |u_0 \overline{\beta} | dv +\int_\Omega |u_0 \overline{P (\beta)}| dv.
$$
By Lemma \ref{finite} and \eqref{L2DF1},
\begin{align*}
\left|\int _{\Omega}u_0 \,\bar{\beta} \, dv\,\right|
&\leq C \Big(\int_{\Omega} |f |^2_{g}(z)   |K(z,a)|^{\alpha} dv_z \Big)^{{1 \over 2}}\Big(\int_{B_a(\epsilon)} |K(z,a)|^{-\alpha} dv_z \Big)^{{1 \over 2}}
 \\
&\leq C \Big(\int_{\Omega} |f |^2_{g} (z) |K(z,a)|^{\alpha} dv_z \Big)^{{1 \over 2}}\Big( v(B_a(\epsilon))K(a)^{-\alpha}\Big{)}^{{1 \over 2}}\\
&\leq C  v(B_a(\epsilon))^{{1+\alpha \over 2}} \Big(\int_{\Omega} |f |^2_{g} (z) |K(z,a)|^{\alpha} dv_z \Big)^{{1 \over 2}}.
\end{align*}
On the other hand, by \eqref{L2DF2},
\begin{eqnarray*}
 \Big{|}\int _{\Omega}u_0 \overline{P(\beta)} dv \Big{|} &\leq &
 C \Big(\int_{\Omega} |f |^2_{g} (z) |K(z,a)|^{\alpha} dv_z \Big)^{{1 \over 2}}v(B_a(\epsilon))  \Big(\int_\Omega \max_{w\in \overline{B_a(\epsilon)}}|K(z, w)|^2 e^{\psi_{a}} dv_z\Big)^{{1 \over 2}}
 \\ &\leq &
  C  \Big(\int_{\Omega}  |f |^2_{g} (z) |K(z,a)|^{\alpha} dv_z \Big)^{{1 \over 2}} v(B_a(\epsilon))   \Big(\int_{\Omega} |K(z,a)|^{2-\alpha}  dv_z\Big)^{{1 \over 2}}.
\end{eqnarray*}
If $\alpha> 1+\frac{(r-1){\bf a}}{2p}\ge 1$, then $|K(z,a)|^{2-\alpha}$ is integrable on $\Omega$ by Theorem \ref{Thm-F-K}. Therefore, for any $a \in \Omega$,
$$
\frac{1}{ v(B_a(\epsilon))} \int \limits_{B_a(\epsilon)}|u| dv \leq C \Big(\int_{\Omega} |f |^2_{g} (z) |K(z,a)|^{\alpha} dv_z \Big)^{{1 \over 2}}.
$$
Coupling this estimate with Proposition \ref{Cauchy--Pompeiu},  one has proved $u$ is bounded. \end{proof}

\section{Sharpness of the pointwise estimates} 

For the Cartan classical domains, we show that the logarithm of the Bergman kernel has a bounded gradient with respect to the Bergman metric, and also verify that Theorem \ref{key} is sharp.

\subsection{Solutions with logarithmic growth}

\begin{example} Let $\Omega$ be a Cartan classical domain and $u(z) = \log K(z)$.  Then $P[u](z)$ is a constant function on $\Omega$ and there exists a constant $c$ so that $|\dbar u|^2_g = c\Tr(zz^*)$.
\end{example}

\begin{proof}  Notice that for all $z \in \Omega$,
\begin{align*}
P[u](z)&= \int_\Omega  u(w) K(z,w) dv_w
\\
&= \int_\Omega{1\over 2\pi} \int_0^{2\pi}  u(e^{i\theta} w) K(z, e^{i\theta} w) d \theta dv_w
\\
&=\int_\Omega{1\over 2\pi} \int_0^{2\pi}  u( w) K(z, e^{i\theta} w) d \theta dv_w
\\
&=\int_\Omega  u( w) K(z, 0) dv_w
\\
&= { 1\over v(\Omega)}\int_\Omega u(w) d v_w,
\end{align*}
where the third equality follows by the transformation rule of the Bergman kernel, and the fourth equality follows by the mean-value property of (anti-)holomorphic functions.

\bigskip
Now, we show the second part of the example. For $z \in M_{(m, n)}(\mathbb{C})$, define $V(z) := I_m - zz^{*}$ and let $V_{uv}$ denote the $(u, v)$ entry of $V$. Then, by \cite{Hua63, Lu97} (c.f. \cite[Proposition 2.1]{Ch-L19}), for domains of type I, II and III,
\[
g^{j\alpha, \overline{k\beta}}(z) = \begin{cases}
V_{jk}(\delta_{\alpha \beta} - \sum_{l=1}^m z_{l\alpha}\overline{z}_{l\beta}), & z \in \hbox{I}(m , n);\\
V_{jk}{V_{\alpha \beta} \over (2 - {\delta_{j\alpha} })(2 - {\delta_{k\beta} })}, & z \in \hbox{II}(n);\\
{ {1 \over 4}V_{jk}V_{\alpha \beta}(1 - \delta_{j\alpha})(1 - \delta_{k\beta})}, & z \in \hbox{III}(n).
\end{cases}
\] 
For matrices $E_{j \alpha} := (\delta_{j u}\delta_{\alpha v})_{u, v}, A := (a_{uv})_{u, v} \in M_{(n, m)}(\mathbb{C})$, it holds that
$$
E_{j \alpha }A = (\delta_{j u} a_{\alpha v})_{u, v} \hbox{ \quad and  \quad } {\partial V \over \partial z_{j\alpha}} = -E_{j\alpha}z^*.
$$
Then for $z \in \hbox{I}(m, n)$,
\begin{align*}
{\partial \log \det V(z) \over \partial z_{j\alpha}}=& \Tr(V^{-1}(z) {\partial V(z) \over \partial z_{j\alpha}}) = - \Tr(V^{-1}(z)E_{j\alpha}z^{*}) \\ =& -\Tr(E_{j\alpha} z^{*} V^{-1}(z)) = - \sum_{u} \delta_{ju}[z^{*}V^{-1}]_{\alpha u} = -[z^{*}V^{-1}]_{\alpha j}. 
\end{align*}
Since $u(z) = \log (\det (V(z)))^{-(m + n)}-\log  v$(I(m, n)) is real-valued,
\begin{align*}
|\overline{\partial}u|^2_g(z)
&= \sum_{j, \beta, k, \alpha} g^{j\alpha, \overline{k\beta}} {\partial u \over \partial z_{j\alpha}} \overline{{\partial u \over \partial z_{k\beta}}}= (m + n)^2 \sum_{j, \beta, k, \alpha}V_{jk}[I - 
z^{\tau}\overline{z}]_{\alpha\beta}[z^{*}V^{-1}(z)]_{\alpha j}
\overline{[z^{*}V^{-1}]}_{\beta k} \\
&= (m + n)^2 \sum_{\alpha, k}[z^{*}]_{\alpha k}[(I - z^{\tau}
\overline{z})z^{\tau}\overline{V^{-1}}]_{\alpha k} = (m + n)^2 \sum_{\alpha, k}[z^{*}]_{\alpha k}[z^{\tau}]_{\alpha k} = 
(m + n)^2 \Tr(zz^{*}).
\end{align*}
For $z \in $II$(n)$, using the symmetry of $z$, we know
$$
{\partial V(z) \over \partial z_{j\alpha}} = -(1 - {\delta_{j\alpha} \over 2})(E_{j\alpha} + E_{\alpha j})z^{*} \hbox{\quad   and  \quad } z^{*}V^{-1}(z) = (z^{*}V^{-1}(z))^{\tau}.
$$
Hence,
\begin{align*}
{\partial \log \det V(z) \over \partial z_{j\alpha}} =& \Tr(V^{-1}(z){\partial V \over \partial z_{j\alpha}}(z))  = -(1 - {\delta_{j\alpha} \over 2})\Tr(E_{j\alpha}z^{*}V^{-1}(z) + z^{*}V^{-1}(z)E_{\alpha j}) \\ 
=& -(2 - \delta_{j\alpha})\Tr(E_{j\alpha}z^{*}V^{-1}(z))  = -(2 - \delta_{j\alpha})[z^{*}V^{-1}(z)]_{j \alpha}.
\end{align*}
Since $u(z) = \log (\det (V(z)))^{-(n + 1)}-\log  v$(II(n)), and for $z$ symmetric $\overline{z^*V(z)^{-1}} = V(z)^{-1}z$,
\begin{align*}
|\overline{\partial}u|_g^2(z)
&= \sum_{j, \beta, k, \alpha} g^{j\alpha, \overline{k\beta}} {\partial u \over \partial z_{j\alpha}} \overline{{\partial u \over \partial z_{k\beta}}}
\\
&= (n+1)^2 \sum_{\alpha, \beta, k, j} V_{jk}V_{\alpha \beta}[z^{*}V^{-1}(z)]_{j \alpha}[\overline{z^{*}V^{-1}(z)}]_{k\beta}
\\
&=
 (n+1)^2 \sum_{j, \beta} [z^{*}]_{j\beta}[V(z)V^{-1}(z)z]_{j\beta} 
=
 (n+1)^2 \Tr(zz^{*}).
\end{align*}
The proof for skew-symmetric $z \in $ III(n) is similar to the preceding proofs. 
\newline
\par
For a Cartan classical domain IV$(n)$, let $s(z) := \sum z_j^2$ and $r(z): = 1 - 2|z|^2 + |s(z)|^2$ for $z \in \mathbb{C}^n$. By \cite{Hua63}, the Bergman kernel $K(z,z)=c r(z)^{-n}$. Also,
$$
g^{j, \overline{k}}(z) = r(z)(\delta_{jk} - 2z_j\overline{z}_k) + 2(\overline{z_j} - \overline{s(z)}z_j)(z_k - s(z)\overline{z_k}).
$$
Notice that
$$
 (\log (r(z)^{-n}))_{z_j} (\log (r(z)^{-n}))_{z_{\bar{k}}} = {4n^2 \over r(z)^2}[z_js(\overline{z}) - \overline{z}_j][\overline{z}_ks(z) - z_k],
$$
\begin{align*}
|\dbar u|_g^2(z)&=4n^2\sum_{j, k = 1}^n [r(z)(\delta_{jk} - 2z_j\overline{z}_k)+ 2(\overline{z}_j - s(\overline{z})z_j)(z_k - s(z)\overline{z}_k)]{(\overline{z}_j - s(\overline{z})z_j)(z_k - s(z)\overline{z}_k)\over r(z)^2}\\
 &= \sum_{j,k  = 1}^n {4n^2\over r(z) }(\delta_{jk} - 2z_j\overline{z_k})[\overline{z}_j - \overline{s(z)}z_j][z_k - s(z) \overline{z}_k] +  \sum_{j, k = 1}^n {8n^2(\overline{z}_j - \overline{s(z)}z_j)^2(z_k - s(z)\overline{z_k})^2\over r(z)^2}\\
& =:F(z)+G(z).
\end{align*}
Thus,
\begin{align*}
F(z) {r \over 4n^2}
&= \sum_{j = 1}^n |z_j|^2 - s \overline{z_j}^2 - z_j^2 \overline{s} + |s|^2|z_j|^2 -2 \sum_{j,k = 1}^n |z_j|^2|z_k|^2 - s|z_j|^2\overline{z_k}^2 - z_j^2 |z_k|^2\overline{s} + |s|^2z_j^2\overline{z_k}^2 \\
&= |z|^2 - 2|s|^2   + |s|^2|z|^2- 2(|z|^4 - s|z|^2\overline{s} - s|z|^2\overline{s} + |s|^2s\overline{s}) \\
&=-2|z|^4 + 5|s|^2|z|^2 - 2|s|^2 + |z|^2 - 2|s|^4,
\end{align*}

$$
G(z) =  {8n^2 \over r^2} |\sum_{j=1}^n (z_j - s\overline{z_j})^2 |^2 = {8n^2 \over r^2} |s - 2s|z|^2 + s^2\overline{s}|^2 = 8n^2|s|^2.
$$
Therefore,
\begin{align*}
|\dbar u|_g^2 &=  {4n^2 \over r}[-2|z|^4 + 5|s|^2|z|^2 - 2|s|^2 + |z|^2 - 2|s|^4] + {4n^2 \over r}2|s(z)|^2r(z) \\
&= {4n^2 \over r}[-2|z|^4 + |z|^2|s|^2 + |z|^2] = 4n^2|z|^2=4n^2 \Tr(z z^*). 
\end{align*}
\end{proof}

Example 2 below shows that the canonical solution to the equation $\bar\partial u=f:=\bar \partial \log K(z)$ (here $\|f\|_{g, \infty}<\infty$) given by $\log K(z) -C_{\Omega}$ is unbounded with logarithmic growth near the boundary of the polydisc.

\begin{example} Consider
$
f(z):=-\sum_{j=1}^n {z_j    (1-|z_j|^2)^{-1}d\zbar_j}
$
defined on $\mathbb{D}^n$. Then, $f$ is $\dbar$-closed, $\|f\|_{g,\infty}\le {1\over 2}$ and the canonical solution to $\dbar u=f$ on $\mathbb{D}^n$ is
\begin{equation} \label{u-poly}
u(z):=\sum_{j=1}^n \log (1-|z_j|^2)-n \int_0^1  \log(1-r) dr.
\end{equation}
\end{example}

\begin{proof} We compute directly that $u$ given by \eqref{u-poly} satisfies $\dbar u=f$, and 
$$
|f(z)|_g^2={1\over2} \sum_{j=1}^n {(1-|z_j|^2)^2 \over (1-|z_j|^2)^2} |z_j|^2 ={|z|^2\over 2}.
$$
\begin{align*}
&P_{\mathbb D^n}\left [\sum_{j=1}^n \log (1-|w_j|^2) \right ](z) ={ 1 \over \pi^n} \int_{{\mathbb D^n}} \prod_{j=1}^n {1\over (1-\langle z_j, w_j\rangle)^2} \sum_{k=1}^n \log (1-|w_k|^2) dv_{w_1}\cdots dv_{w_n}\\
=& \sum_{k=1}^n {1\over \pi } \int_{{\mathbb D^n}}  {\log (1-|w_k|^2)\over (1-\langle  z_k , w_k \rangle)^2}  d v_{w_k} =\sum_{k=1}^n  2 \int_0^1   \log (1-r_k^2)  r_k dr_k =n  \int_0^1   \log (1-r)  dr. 
\end{align*}
\end{proof}

\subsection{A sharp example}  
The maximum blow-up order for a solution to $\dbar u = f$ with $\|f\|_{g, \infty} < \infty$ is $\int_{\Omega}|K(\cdot, w)| dv_w$. We will provide an example here to show that Theorem \ref{key} is sharp on the Cartan classical domains.

\begin{pro} \label{eg}  Let $\Omega$ be a Cartan classical domain. 
Then, there is a constant $c$ such that for each $z\in \Omega$, there is a $\dbar$-closed $(0,1)$-form $f_z$ on $\Omega$ with
$\|f_z\|_{g, \infty} = 1$ and the canonical solution to $\dbar u=f_z$ satisfies 
$$
|u(z)| \geq  {c} \int_{\Omega} |K(z, w)| dv_w.
$$ 
\end{pro}
 
\begin{proof} 
For any point $z \in \Omega$, consider the function $U_z(\cdot):=K(\cdot)^{-1} K(\cdot, z)$ and
$$
f_z(\cdot) :=\bar \partial U_z(\cdot) =K(\cdot, z) \bar \partial (K(\cdot)^{-1}).
$$
Then, by Example 1, 
\begin{eqnarray*}
\| f_z\|_{g, \infty} &=& \|K(\cdot, z)  K(\cdot)^{-2} \bar \partial (K(\cdot))\|_{g, \infty} \\
&=&\| |K(\cdot, z)K(\cdot)^{-1}  \bar \partial (\log K(\cdot)) \|_{g, \infty} \\
&\le&  \|K(\cdot, z)K(\cdot)^{-1} \|_\infty \|   \bar \partial (\log K(\cdot))| \|_{g, \infty} \\
&\le& C.
\end{eqnarray*}
The Bergman projection of $U_z$ is 
$$
P[U_z] (\cdot) =  \int_{\Omega} U_z (w) K(\cdot, w) dv_w=  \int_{\Omega} K(w)^{-1} K(w, z) K(\cdot, w) dv_w.
$$
In particular, by \eqref{J_0} with $\beta=-1$  
\begin{eqnarray*}
P[U_z] (z) &= & \int_{\Omega} K(w)^{-1} K(w, z) K(z, w) dv_w\\
&=&\int_{\Omega} K(w)^{-1} |K(w, z)|^2 dv_w\\
&\approx &\int_\Omega |K(z, w)| dv_w.
\end{eqnarray*}
The canonical solution to $\bar \partial u=f$ is $u_z:=U_z-P[U_z]$ and
$$
|u_z(z)| = |1 - J_{-1, 0}(z)|\ge { c} \int_\Omega |K(z,w)| dv_w -1
$$
for a uniform constant $c>0$, independent of $z$.
\end{proof}

\subsection{Blow-up order greater than $\log$}

With the previous example and Theorem \ref{Thm-LZ} we will provide the maximum blow-up order when $\Omega$ is a Cartan classical domain of rank 2. By Theorem \ref{Thm-LZ}, for $z = te_1 + te_2$ where $e_1, e_2 \in \mathcal{U}$,
$$
\int_\Omega |K(z,w)| dv_w \approx (1-t)^{-{{\bf a} \over 2}} \approx \delta_{\Omega}(z)^{-{\bf a} \over 2}, \quad \hbox{ as } t\to 1^{-}.
$$
When $\Omega$ is IV$(n)$ with $n \geq 3$, 
$$
\int_\Omega |K(z,w)| dv_w  \approx \delta_{\Omega}(z)^{- {n \over 2} + 1}.
$$
When $\Omega$ is III(4) or III(5), 
$$
\int_\Omega |K(z,w)| dv_w  \approx \delta_{\Omega}(z)^{-2}.
$$
When $\Omega$ is I(2, n) with $n \geq 2$,
$$
\int_\Omega |K(z,w)| dv_w  \approx \delta_{\Omega}(z)^{-1}.
$$
When $\Omega$ is II(2),
$$
\int_\Omega |K(z,w)| dv_w  \approx \delta_{\Omega}(z)^{-{1 \over 2}}.
$$

\subsection*{Acknowledgements} \small
The first author sincerely thanks Professors Bo-Yong Chen and Jinhao Zhang for their suggestions and warm encouragement throughout the years.

We greatly appreciate the referees who read our original version very carefully and raised many valuable questions, which were very helpful for us when we revised our paper.

\bibliographystyle{alphaspecial}

\begingroup

\fontsize{10}{10.5}\selectfont

\fontsize{11}{11}\selectfont
 \bigskip
 
\noindent Department of Mathematics, University of Connecticut, Storrs, CT 06269-1009, USA

 \medskip
 
 \noindent dong@uconn.edu

\vspace{0.2 cm}

\noindent

  \bigskip
 
\noindent Department of Mathematics, University of California, Irvine, CA 92697-3875, USA

 \medskip
\noindent sli@math.uci.edu

 \bigskip
 
\noindent Department of Mathematics, University of California, Irvine, CA 92697-3875, USA

 \medskip
\noindent jtreuer@uci.edu
 
\end{document}